\documentclass{amsart}[12pt]
\usepackage{amssymb,amsmath}
\usepackage{enumerate}
\usepackage[english]{babel}
\usepackage{color}

\newcommand{\Om} {\Omega}
\newcommand {\ep} {\varepsilon}
\newcommand {\om} {\omega}
\newcommand {\gm} {\gamma}
\newcommand {\sg} {\sigma}
\newcommand {\ii} {\infty}
\newcommand {\dt} {\delta}

\newcommand {\lb} {\lambda}

\newcommand {\ol} {\overline}
\newcommand {\sm} {\setminus}
\newcommand {\su} {\subset}
\newcommand {\wt} {\widetilde}

\newcommand {\pr} {\prime}
\newcommand {\mc} {\mathcal}
\newcommand {\mbb} {\mathbb}

\newtheorem{teo}{Theorem}[section]
\newtheorem{pro}{Proposition}[section]
\newtheorem{cor}{Corollary}[section]

\theoremstyle{definition}
\newtheorem{rem}{Remark}[section]
\newtheorem{df}{Definition}[section]

\title{Almost Uniform Convergence\\ in Wiener-Wintner Ergodic Theorem}

\author{Vladimir Chilin}
\address{National University of Uzbekistan\\ Tashkent, 100174, Uzbekistan}
\email{vladimirchil@gmail.com; chilin@ucd.uz}

\author{Semyon Litvinov}
\address{Pennsylvania State University\\ 76 University Drive\\ Hazleton, PA 18202, USA}
\email{snl2@psu.edu}

\begin{document}
\begin{abstract}
We extend almost everywhere convergence in Wiener-Wintner ergodic theorem for $\sg$-finite measure to a generally stronger almost uniform convergence and present a larger, universal, space for which this convergence holds.  We then extend this result to the case with Besicovitch weights.
\end{abstract}
\date{March 23, 2020}

\maketitle
\section{Introduction and Preliminaries}
Let $(\Om,\mu)$ be a measure space. Denote by $\mc L^0$ the algebra of almost everywhere (a.e.) finite complex-valued measurable functions on $(\Om,\mu)$, and let $\mc L^p\su \mc L^0$, $1\leq p\leq\ii$, stand for the $L^p$-space on $(\Om,\mu)$ equipped with the standard norm $\|\cdot\|_p$.

A sequence $\{f_n\}\su\mc L^0$ is said to converge {\it almost uniformly (a.u.)} if there is $\widehat f\in\mc L^0$ such that, given $\ep>0$, there exists $\Om^\pr\su\Om$ satisfying conditions
\[
\mu(\Om\sm\Om^\pr)\leq\ep\text{ \ \ and \ \ }\lim_{n\to\ii}\|(\widehat f-f_n)\chi_{\Om^\pr}\|_\ii=0.
\]
It is clear that $\{f_n\}\su\mc L^0$ converges a.u. if and only if for every $\ep>0$ there exists $\Om^\pr\su\Om$ such that
\[
\mu(\Om\sm\Om^\pr)\leq\ep\text{\ \ and the sequence\ \,}\{f_n\chi_{\Om^\pr}\}\text{\ \,converges in\ \,}\mc L^\ii,
\]
that is, this sequence converges uniformly.

It is easy to see that if the measure $\mu$ is not finite, a.u. convergence is generally stronger than a.e. convergence, whereas, due to Egorov's theorem, these convergences coincide when $\mu(\Om)<\ii$.

Let $T:\Om\to\Om$ be a measure preserving transformation (m.p.t.). Given \\$\lb\in\mbb C_1=\{\lb\in\mbb C: |\lb|=1\}$ and $f\in\mc L^0$, denote
\[
M_n(T,\lb)(f)=\frac1n\sum_{k=0}^{n-1}\lb^kf\circ T^k\text{ \ \ and \ \ }M_n(T)(f)=\frac1n\sum_{k=0}^{n-1}f\circ T^k.
\]

\begin{df}
We write $f\in a.e.\,WW(\Om,T)$ ($f\in a.u.\,WW(\Om,T)$) if
\[
\exists\, \ \Om_f\su\Om \text{ \ with \ } \mu(\Om\sm\Om_f)=0 \text{ \ such that the sequence\ \,}\big\{M_n(T,\lb)(f)(\om)\big\}
\]
converges for any $\om\in\Om_f$ and $\lb\in\mbb C_1$.

\noindent
(respectively, if
\[
\forall\, \ \ep>0 \ \ \exists\, \ \Om^\pr=\Om_{f,\ep}\text{\ \,with \ }\mu(\Om\sm\Om^\pr)\leq\ep\text{ \ such that the sequence\ } \big\{M_n(T,\lb)(f)\chi_{\Om^\pr}\big\}
\]
converges uniformly for any $\lb\in\mbb C_1$).
\end{df}

Clearly, we have $a.u.\,WW\su a.e.\,WW$, but even if $\mu(\Om)<\ii$, Egorov's theorem does not entail the opposite inclusion.

The celebrated Wiener-Wintner theorem \cite{ww} asserts that $\mc L^1\su a.e.\,WW$, provided $\mu(\Om)<\ii$.
It is shown in \cite[Theorem 2.10]{as} that for a {\it uniquely ergodic system}, that is, when $\mu$ is the only invariant measure for $T$, and a continuous function the convergence in Wiener-Wintner theorem is uniform in $\Om$; a related result was previously obtained in \cite{ro}. For a review of results on uniform convergence in Wiener-Wintner-type ergodic theorems for uniquely ergodic systems and continuous functions, see \cite{ez}.

Furthermore, Assani's extension of Bourgain's Return Times theorem \cite[Theorem 5.1]{as} entails that if $(\Om,\mu)$ is $\sg$-finite, then $\mc L^p\su a.e.\,WW$, $1\leq p<\ii$.

Now, let  $(\Om,\mu)$ be a $\sg$-finite measure space, and let $T:\Om\to\Om$ be a m.p.t. The main goal of this article is to show that if $T$ is ergodic, then $\mc R_\mu\su a.u.\,WW(\Om,T)$. Here, $\mc R_\mu$ - which coincides with $\mc L^1$ if $\mu(\Om)<\ii$ - is a universal, relative to a.u. convergence of the averages $M_n(T,\lb)$, space that contains not only every space $\mc L^p$ for $1\leq p<\ii$ but also classical Banach spaces on $(\Om,\mu)$ such as Orlicz, Lorentz, and Marcinkiewicz spaces $X$ with $\chi_\Om\notin X$. \ Thus, by relaxing uniform convergence to almost uniform convergence, we gain convergence for a much wider class of functions than the class of continuous functions and without the assumption of finiteness of measure. Then we further generalize this result by expanding the family $\big\{ \{\lb^k\}:\,\lb\in\mbb C_1\big\}$ to the class of all bounded Besicovitch sequences.

In what follows, we reduce the problem to showing that $\mc L^1\su a.u.\,WW$, which in turn can be derived from the case $\mu(\Om)<\ii$ with the help of Hopf decomposition. The following Corollary \ref{c11}, a consequence of the maximal ergodic inequality, further reduces the problem to finding a set $\mc D\su a.u.\,WW$ that is dense in $\mc L^1$. To this end, we take the path of "simple inequality" as outlined in \cite{as} and employ Egorov's theorem and a form of Van der Corput's inequality to show that $\mc D=\mc L^2\su a.u.\,WW$.

For $\ol b=\{b_k\}_{k=0}^\ii\su\mbb C$, denote
\[
M_n(T,\ol b)(f)=\frac1n\sum_{k=0}^{n-1}b_kf\circ T^k.
\]
Let $\mc B$ be a subset of the set of bounded sequences $\ol b=\{b_k\}_{k=0}^\ii\su\mbb C$.

\begin{pro}\label{p11}
Let $(\Om,\mu)$ be $\sg$-finite, and let $1\leq p<\ii$. Then the set
\[
\mc C_p(\mc B)=\Big\{f\in\mc L^p:\ \forall \ \ep>0\ \ \exists \ \Om^\pr\su\Om\text{\ with\ }\mu(\Om\sm\Om^\pr)\leq\ep\text{\ such that}
\]
\[
\ \ \ \ \ \ \ \ \ \ \ \ \ \ \ \ \ \ \ \ \text{the sequence\ \,} \{M_n(T,\ol b)(f)\chi_{\Om^\pr}\}\text{\ converges uniformly\ \,}\forall \ \ol b\in\mc B\Big\}
\]
is closed in $\mc L^p$.
\end{pro}
\begin{proof}
Given $l\in\mbb N$, denote
\[
\mc B_l=\big\{\{b_k\}\in\mc B:\, |b_k|\leq l \ \ \forall\ k\big\}.
\]
Let $\{f_k\}\su\mc C_p(\mc B)$ and $f\in\mc L^p$ be such that $\|f-f_k\|_p\to 0$. Given $\ep>0$ and $\dt>0$, the maximal ergodic inequality
\[
\mu\left\{\sup\limits_nM_n(T)(|g|)>t\right\}\leq\left(2\frac{\|g\|_p}t\right)^p \ \ \forall \ \,g\in\mc L^p, \ t>0
\]
(see, for example, \cite{cl}) entails that there exists $f_{k_0}$ for which
\begin{equation}\label{e11}
\mu\left\{\sup\limits_nM_n(T)(|f-f_{k_0}|)>\frac\dt{3l}\right\}\leq\frac\ep{2^{l+1}}.
\end{equation}
Next, as
\[
|M_n(T,\ol b)(f-f_{k_0})|\leq l\, M_n(T)(|f-f_{k_0}|)\, \ \ \forall \ \ol b\in\mc B_l,
\]
inequality (\ref{e11}) implies that
\[
\mu\left\{\sup_n|M_n(T,\ol b)(f-f_{k_0})|>\frac\dt3\right\}\leq\frac\ep{2^{l+1}}\ \ \ \forall \ \ol b\in\mc B_l.
\]
Therefore, with
\[
\Om_{l,1}=\left\{\sup_n|M_n(T,\ol b)(f-f_{k_0})|\leq\frac\dt3\right\},
\]
we have $\mu(\Om\sm\Om_{l,1})\leq\displaystyle\frac\ep{2^{l+1}}$ and
\[
\|M_n(T,\ol b)(f-f_{k_0})\chi_{\Om_{l,1}}\|_\ii\leq\frac\dt3 \ \ \ \forall \ n\in\mbb N, \ \ol b\in\mc B_l.
\]
Now, letting $\Om_1=\bigcap_{l=1}^\ii\Om_{l,1}$, we obtain $\mu(\Om\sm\Om_1)\leq\displaystyle\frac\ep2$ and
\[
\|M_n(T,\ol b)(f-f_{k_0})\chi_{\Om_1}\|_\ii\leq\frac\dt3 \ \ \ \forall \ n\in\mbb N\text{ \ and \ } \ol b\in\mc B
\]

Furthermore, as $f_{k_0}\in\mc C(\mc B)$, there exists $\Om_2\su\Om$ with $\mu(\Om\sm\Om_2)\leq\displaystyle\frac\ep2$ such that the sequence $\{M_n(T,\ol b)(f_{k_0})\chi_{\Om_2}\}$ converges uniformly for all $\ol b\in\mc B$. Thus, for each $\ol b\in\mc B$, there is a number $N=N(\ol b)$ such that
\[
\left\|(M_m(T,\ol b)(f_{k_0})-M_n(T,\ol b)(f_{k_0}))\chi_{\Om_2}\right\|_\ii\leq\frac\dt3 \ \ \ \forall \ m,n\ge N.
\]
Now, setting $\Om^\pr=\Om_1\cap\Om_2$, we have $\mu(\Om\sm\Om^\pr)\leq\ep$ and, for each $\ol b\in\mc B$ and all $m,n\ge N(\ol b)$,
\[
\begin{split}
\|(M_m(T,\ol b)(f)&-M_n(T,\ol b)(f))\chi_{\Om^\pr}\|_\ii\leq\|M_m(T,\ol b)(f-f_{k_0})\chi_{\Om_1}\|_\ii\\
&+\|(M_m(T,\ol b)(f_{k_0})-M_n(T,\ol b)(f_{k_0}))\chi_{\Om_2}\|_\ii\\
&+\|M_n(T,\ol b)(f-f_{k_0})\chi_{\Om_1}\|_\ii\leq\dt,
\end{split}
\]
implying that the sequence $\{M_n(T,\ol b)(f)\chi_{\Om^\pr}\}$ converges uniformly for all $\ol b\in\mc B$, hence $f\in\mc C_p(\mc B)$.
\end{proof}

\begin{cor}\label{c11}
Let $(\Om,\mu)$ be $\sg$-finite. Then $\mc L^p\cap a.u.\,WW$ is closed in $\mc L^p$ for each $1\leq p<\ii$.
\end{cor}
\begin{proof}
Apply Proposition \ref{p11} to $\mc B=\{ \ol b=\{\lb^k\}: \,\lb\in\mbb C_1\}$.
\end{proof}

Next, let $\mc K$ be the $\|\cdot\|_2$-closure of the linear span of the set
\[
K=\left\{f\in\mc L^2:\, f\circ T=\lb_ff\text{ \ for some\, }\lb_f\in\mbb C_1\right\}.
\]

\begin{pro}\label{p12}
$\mc K\su a.u.\,WW(\Om,T)$.
\end{pro}
\begin{proof}
By Corollary \ref{c11}, it is sufficient to show that $\sum\limits_{j=1}^mz_jf_j\in a.u.\,WW$ whenever $z_j\in\mbb C$ and $f_j\in K$ for all $1\leq j\leq m$. This, in turn, will easily follow from $K\su a.u.\,WW$.

So, pick $f\in K$ and $\ep>0$. Then there exists $\Om^\pr=\Om_{f,\ep}$ such that $\mu(\Om\sm\Om^\pr)\leq\ep$ and
$f\chi_{\Om^\pr}\in\mc L^\ii$. In addition, given $\lb\in\mbb C_1$, we have
\[
M_n(T,\lb)(f)\chi_{\Om^\pr}=\frac1n\sum_{k=0}^{n-1}(\lb\lb_f)^k\chi_{\Om^\pr}.
\]
Therefore, since the sequence $\left\{\frac1n\sum_{k=0}^{n-1}(\lb\lb_f)^k\right\}$ converges in $\mbb C$, we conclude that the averages $M_n(T,\lb)(f)\chi_{\Om^\pr}$ converge uniformly, hence $f\in a.u.\,WW$.
\end{proof}

\section{The case of finite measure}

Let $(\Om,\mu)$ be a finite measure space, and let $T$ be an m.p.t. If $Uf=f\circ T$, $f\in\mc L^0$, then $U: \mc L^2\to\mc L^2$ is a surjective linear isometry with $U^*=U^{-1}$. Given $f,g\in\mc L^2$, denote $(f,g)=\int_\Om\ol fgd\mu$, an inner product in the Hilbert space $\mc L^2$.

If $f\in\mc L^2$ and $l\in\mbb Z$, define
\[
\gm_f(l)=
\begin{cases}
\ol{(f,U^{-l}f)}&\text{if \ $ l<0$} \\
(f,U^lf)&\text{if \ $l \ge 0$}.
\end{cases}
\]
It is easily verified that the sequence $\{\gm(l)\}_{-\ii}^\ii$ is positive definite, that is, for any $z_0,\dots,z_m\in\mbb C$,
\[
\sum_{i,j=0}^m\gm(i-j)z_i\ol z_j\ge0.
\]
Therefore, Herglotz-Bochner theorem implies that there exists a positive Borel measure $\sg_f$ on $\mbb C_1$ such that
\[
\int_\Om\ol f\cdot(f\circ T^l)d\mu=(f,U^lf)=\gm_f(l)=\widehat\sg_f(l)=\int_{\mbb C_1}e^{2\pi il\lb}d\sg_f(\lb), \ \ l=1,2,\dots
\]

Let now $\mc K^\perp$ be the orthogonal compliment of $\mc K$ in the Hilbert space $\mc L^2$. It is known that if $f\in\mc K^\perp$, then the measure $\sg_f$ is continuous; see, for example, \cite[p.\,27]{as}. Let us provide an independent proof of this claim. We will need the following.

\begin{pro}\label{p21}
$U(\mc K^\perp)\su\mc K^\perp$.
\end{pro}
\begin{proof}
Since $U^*=U^{-1}$, it follows that $U^*f = \lb_f^{-1}f$ for all $0\neq f\in K$. Thus, given $g\in\mc K^\perp$ and $f\in K $, we have
\[
(Ug,f)=(g,U^*f)=\lb_f^{-1}(g,f)=0,
\]
hence  $Ug\in\mc K^\perp$.
\end{proof}

\begin{pro}\label{p22}
If $f\in\mc K^\perp$, then $\sg_f$ is a continuous measure, that is, $\sg_f\{\lb\}=0$ for every $\lb\in\mbb C_1$.
\end{pro}
\begin{proof}
It is known \cite[p.\,42]{ka} that
\[
\sg_f\{\lb\}=\lim_{n\to\ii}\frac1n\sum_{l=1}^ne^{2\pi il\lb}\widehat\sg_f(\lb).
\]
Therefore, we have
\[
\begin{split}
\sg_f\{\lb\}&=\lim_{n\to\ii}\frac1n\sum_{l=1}^ne^{2\pi il\lb}\int_\Om\ol f\cdot(f\circ T^l)d\mu\\
&=\lim_{n\to\ii}\int_\Om\ol f\cdot\frac1n\sum_{l=1}^ne^{2\pi il\lb}f\circ T^ld\mu,
\end{split}
\]
thus, it would be sufficient to verify that
\begin{equation}\label{e21}
\lim_{n\to\ii}\left\|\frac1n\sum_{l=1}^ne^{2\pi il\lb}f\circ T^l\right\|_2=0.
\end{equation}
Mean Ergodic theorem for $\widetilde U: \mc L^2\to\mc L^2$ given by $\widetilde Uf=e^{2\pi i\lb}Uf$ implies that there is
$\widehat f\in\mc L^2$ such that
\[
\lim_{n\to\ii}\left\|\frac1n\sum_{l=1}^ne^{2\pi il\lb}f\circ T^l-\widehat f\,\right\|_2=0
\]
By Proposition \ref{p21}, $f\circ T^l\in\mc K^\perp$ for each $l$, so $\widehat f\in\mc K^\perp$. Also,
\[
\widehat f\circ T=\|\cdot\|_2-\lim_{n\to\ii}\frac1n\sum_{l=1}^ne^{2\pi il\lb}f\circ T^{l+1}=e^{-2\pi i\lb}\widehat f,
\]
so that $\widehat f\in\mc K$. Therefore $\widehat f=0$, and (\ref{e21}) follows.
\end{proof}

Now we shall turn our attention to a case of Van der Corput's Fundamental Inequality. Let $n\ge 1$ and $0\leq m\leq n-1$ be integers, and let $f_0,\dots,f_{n-1+m}\in\mc L^0$ be such that $f_n=\dots=f_{n-1+m}=0$. Then, replacing in the inequality in \cite[Ch.1, Lemma 3.1]{kn} $N$ by $n+1$ and $H$ by $m+1$, we obtain
\[
\begin{split}
\left|\frac1n\sum_{k=0}^{n-1}f_k\right|^2&\leq\frac{n+m-1}{n(m+1)}\frac1n\sum_{k=0}^{n-1}|f_k|^2\\
&+\frac{2(n+m-1)}{n(m+1)}\sum_{l=1}^m\frac{m+1-l}{m+1}\operatorname{Re}\frac1n\sum_{k=0}^{n-1}\ol f_kf_{k+l}.
\end{split}
\]
If $f_0,\dots,f_{n-1}\in\mc L^\ii$, then the above inequality entails
\[
\begin{split}
\left\|\frac1n\sum_{k=0}^{n-1}f_k\right\|^2_\ii&
\leq\frac{n+m-1}{n(m+1)}\left\|\frac1n\sum_{k=0}^{n-1}|f_k|^2\right\|_\ii\\&+
\frac{2(n+m-1)}{n(m+1)}\sum_{l=1}^m\frac{m+1-l}{m+1}\left\|\frac1n\sum_{k=0}^{n-1}\ol f_kf_{k+l}\right\|_\ii,
\end{split}
\]
which, in turn implies that
\begin{equation}\label{e22}
\left\|\frac1n\sum_{k=0}^{n-1}f_k\right\|^2_\ii<\frac2{m+1}
\left\|\frac1n\sum_{k=0}^{n-1}|f_k|^2\right\|_\ii+\frac4{m+1}\sum_{l=1}^m\left\|\frac1n\sum_{k=0}^{n-1}\ol f_kf_{k+l}\right\|_\ii.
\end{equation}

Here is the extension of Wiener-Wintner theorem to the case of a.u. convergence and an ergodic measure preserving transformation:
\begin{teo}\label{t21}
If a m.p.t. $T$ is ergodic, then $\mc L^1\su a.u.\,WW(\Om,T)$.
\end{teo}
\begin{proof}
By Corollary \ref{c11}, as $\mc L^2$ is dense in $\mc L^1$,
$\mc L^2=\mc K\oplus\mc K^\perp$, and, by Proposition \ref{p12}, $\mc K\su a.u.\,WW$, it remains to show that $\mc K^\perp\su a.u.\,WW$.

Let $f\in\mc K^\perp$, and $\ep>0$. By the pointwise ergodic theorem, since $T$ is ergodic, we have
\[
M_n(T)(|f|^2)\to\|f\|^2_2\text{ \ a.e.\ \ and \ }M_n(T)(\ol f\cdot(f\circ T^l))\to\widehat\sg_f(l)\text{ \ a.e.\ \, }\forall\, \ l=1,2,\dots
\]
Applying Egorov's theorem repeatedly, we can construct $\Om^\pr=\Om_{f,\ep}\su\Om$ such that $\mu(\Om\sm\Om^\pr)\leq\ep$ and
\begin{equation}\label{e23}
\left\|M_n(T)(|f|^2)\chi_{\Om^\pr}\right\|_\ii\to\|f\|^2_2\text{ \ \ and \ }\left\|M_n(T)(\ol f\cdot(f\circ T^l))\chi_{\Om^\pr}\right\|_\ii\to\widehat\sg_f(l)
\end{equation}
for all $l=1,2,\dots$

If $\lb\in\mbb C_1$ and $f_k=\lb^kf\circ T^k\chi_{\Om^\pr}$, then a simple calculation yields
\[
\ol f_kf_{k+l}=\lb^l(\ol f\cdot(f\circ T^l))\circ T^k\chi_{\Om^\pr},\text{\ hence\ \ }|f_k|^2=|f|^2\circ T^k\chi_{\Om^\pr}, \ \forall \ k,l=0,1,2,\dots
\]
Therefore, inequality (\ref{e22}) implies that
\[
\begin{split}
\sup_{\lb\in\mbb C_1}\left\|M_n(T,\lb)(f)\chi_{\Om^\pr}\right\|_\ii^2&<\frac2{m+1}\left\|M_n(T)(|f|^2)\chi_{\Om^\pr}\right\|_\ii\\
&+\frac4{m+1}\sum_{l=1}^m\left\|M_n(T)(\ol f\cdot(f\circ T^l))\chi_{\Om^\pr}\right\|_\ii.
\end{split}
\]
Thus, for a fixed $m$, in view of (\ref{e23}), we obtain
\[
\limsup_n\sup_{\lb\in\mbb C_1}\left\|M_n(T,\lb)(f)\chi_{\Om^\pr}\right\|_\ii^2\leq\frac2{m+1}\|f\|_2^2+\frac4{m+1}\sum_{l=1}^m|\widehat\sg_f(l)|.
\]
Since, by Proposition \ref{p22}, the measure $\sg_f$ is continuous, Wiener's criterion of continuity of positive finite Borel measure \cite[p.\,42]{ka} yields
\[
\lim_{m\to\ii}\frac1{m+1}\sum_{l=1}^m|\widehat\sg_f(l)|^2=0,\text{ \ hence \ }\lim_{m\to\ii}\frac1{m+1}\sum_{l=1}^m|\widehat\sg_f(l)|=0,
\]
and we conclude that
\[
\limsup_n\sup_{\lb\in\mbb C_1}\left\|M_n(T,\lb)(f)\chi_{\Om^\pr}\right\|_\ii=0,
\]
so $f\in a.u.\,WW(\Om,T)$.
\end{proof}

\begin{rem}
In the classical case, a simple application of the ergodic decomposition theorem yields convergence of the averages $M_n(T,\lb)(f)$, with $f\in\mc L^1$, on a set of full measure for all $\lb\in\mbb C_1$ without assumption of ergodicity of the m.p.t. $T$; see, for example, \cite[Theorem 2.12]{as}. Unfortunately, this does not seem to be the case with a.u. convergence. So, the question whether Theorem \ref{t21} remains valid for a non-ergodic m.p.t. $T$ remains open.
\end{rem}

\section{The case of infinite measure}

Assume now that $(\Om,\mu)$ is $\sg$-finite, while $T:\Om\to\Om$ is an ergodic m.p.t. In the next theorem, we employ the idea of the proof of \cite[Theorem 5.1]{as}.
\begin{teo}\label{t31}
$\mc L^1(\Om)\su a.u.\,WW(\Om,T)$.
\end{teo}
\begin{proof}
Fix $f\in\mc L^1$ and $\ep>0$. Let $\Om=C\cup D$ be the Hopf decomposition, where $C$ is the conservative and $D=\Om\sm C$ the dissipative part of $\Om$. Then, by \cite[\S\,3.1, Theorem 1.6]{kr}, or \cite[\S\,3.7, Theorem 7.4]{pe}, we have
\[
nM_n(T)(|f|)(\om)=\sum_{k=0}^{n-1}|f|(T^k\om)<\ii
\]
for almost all $\om\in D$. Besides, since, by \cite[Theorem 3.1]{cl}, the sequence $\{M_n(T)(|f|)\}$ converges a.u., there is $\Om_1\su D$ such that
\[
\mu(D\sm \Om_1)\leq\frac\ep3\text{ \ \ and \ \ } \{M_n(T)(|f|)\chi_{\Om_1}\}\text{\ \ converges uniformly}.
\]
Then, as $M_n(T)(|f|)\to 0$ a.e. on $D$, it follows that $M_n(T)(|f|)\chi_{\Om_1}\to 0$ uniformly. Therefore, in view of
\[
|M_n(T,\lb)(f)\chi_{\Om_1}|\leq M_n(T)(|f|)\chi_{\Om_1} \ \ \forall \ \lb\in\mbb C_1,
\]
we conclude that
\begin{equation}\label{e31}
M_n(T,\lb)(f)\chi_{\Om_1}\to 0\text{ \ \ uniformly} \ \ \forall \ \lb\in\mbb C_1.
\end{equation}

Next, since $ (\Om,\mu)$ is $\sg$-finite, applying an exhaustion argument, one can construct $p\in\mc L^1_+$ such that
$p\circ T=p$ and $\wt C=\{p>0\}$ is the maximal modulo $\mu$ subset of $C$ on which there exists a finite $T$-invariant measure; see \cite[pp.131,\,132]{kr}. Besides, by \cite[Lemma 3.11, Theorem 3.12]{kr}, $\wt C$ and $C\sm\wt C$ are $U$-absorbing (equivalently, $T$-absorbing) and $M_n(T)(|f|)\to 0$ a.e. on $C\sm \widetilde C$. Hence, as above, there exists a set $\Om_2\su C\sm\wt C$ such that
\[
\mu((C\sm\wt C)\sm \Om_2)\leq\frac\ep3\text{ \ \ and \ \ } M_n(T)(|f|)\chi_{\Om_2}\to 0\text{ \ \ uniformly},
\]
implying that
\begin{equation}\label{e32}
M_n(T,\lb)(f)\chi_{\Om_2}\to 0\text{ \ \ uniformly} \ \  \forall \ \lb\in\mbb C_1.
\end{equation}

If we define $\mu^\pr=p\cdot\mu\sim \mu$, then $\mu^\pr$ is a $U$-invariant (equivalently, $T$-invariant, that is, $p\circ T=p$), finite measure on $\wt C$.
It follows that $T$ is a m.p.t. on the finite measure space $(\widetilde C,\mu^\pr)$:
\[
\mu^\pr(T^{-1}A)=\int_{T^{-1}A}p\,d\mu=\int_Ap\circ T\,d\mu=\int_Ap\,d\mu=\mu^\pr(A).
\]
In addition, as $\wt C$ is $T$-absorbing, ergodicity of $T$ and $\mu^\pr\sim\mu$ entails that $T:\wt C\to\wt C$ is an ergodic m.p.t. Also, since $f\in\mc L^1(\wt C, \mu)$, we have  $fp^{-1}\in\mc L^1(\wt C,\mu^\pr)$. Therefore, by Theorem \ref{t21}, there exists $\Om_3\su\wt C$ such that
\[
\mu(\wt C\sm \Om_3)=\mu^\pr(\wt C\sm\Om_3)\leq\frac\ep3
\]
and the averages
\[
M_n(T,\lb)(fp^{-1})\chi_{\Om_3}=\frac1n\sum_{k=0}^{n-1}\lb^k(fp^{-1})\circ T^k\chi_{\Om_3}
\]
converge uniformly for all $\lb\in\mbb C_1$.
But $(fp^{-1})\circ T^k\chi_{\Om_3}=p^{-1}(f\circ T^k)\chi_{\Om_3}$, and we conclude that the sequence
\begin{equation}\label{e33}
\big\{M_n(T,\lb)(f)\chi_{\Om_3}=pM_n(T,\lb)(fp^{-1})\chi_{\Om_3}\big\}\text{\ \ converges uniformly} \ \ \forall \ \lb\in\mbb C_1.
\end{equation}

Now, with $\Om^\pr=\Om_1\cup \Om_2\cup \Om_3$, in view of (\ref{e31})\,-\,(\ref{e33}), we obtain
\[
\mu(\Om\sm\Om^\pr)\leq\ep\text{ \ \ and \ \ } \{M_n(T,\lb)(f)\chi_{\Om^\pr}\}\text{\ \ converges uniformly} \ \ \forall \ \lb\in\mbb C_1.
\]
Therefore $f\in a.u.\,WW(\Om,T)$, and the proof is complete.
\end{proof}

Denote
\[
\mc R_\mu=\left\{f\in\mc L^1+\mc L^\ii : \ \mu\{|f|>\lb\}<\ii\text{ \ for all \ }\lb>0\right\}.
\]

\begin{teo}\label{t32}
$\mc R_\mu\su a.u.\,WW(\Om,T)$.
\end{teo}
\begin{proof}
Pick $f\in\mc R_\mu$ and fix $\ep>0$, $\dt>0$. By \cite[Proposition 2.1]{ccl}, there exist $g\in\mc L^1$ and $h\in\mc L^\ii$ such that
\[
\|h\|_\ii\leq\frac\dt3\text{ \ and \ } f=g+h.
\]
As $g\in\mc L^1$, Theorem \ref{t31} entails that there exists $\Om^\pr\su\Om$ such that for each $\lb\in\mbb C_1$  there is a number $N=N(\lb)$ satisfying conditions
\[
\mu(\Om\sm\Om^\pr)\leq\ep\text{ \ \ and \ \ }\|(M_m(T,\lb)(g)-M_n(T,\lb)(g))\chi_{\Om^\pr}\|_\ii\leq\frac\dt3\, \ \ \forall\,\ m,n\ge N.
\]
Then, given $\lb\in\mbb C_1$ and $m,n\ge N(\lb)$, we have
\[
\begin{split}
\|(M_m(T,\lb)(f)&-M_n(T,\lb)(f))\chi_{\Om^\pr}\|_\ii\leq\|(M_m(T,\lb)(g)-M_n(T,\lb)(g))\chi_{\Om^\pr}\|_\ii\\
&+\|M_m(T,\lb)(h)\|_\ii+\|M_n(T,\lb)(h)\|_\ii\leq\frac\dt3+2\|h\|_\ii\leq\dt,
\end{split}
\]
implying that $f\in a.u.\,WW$.
\end{proof}

As $\mc L^p\su\mc R_\mu$ for any $1\leq p<\ii$, Theorem \ref{t32} yields the following.
\begin{cor}\label{c31}
$\mc L^p\su a.u.\,WW(\Om,T)$ for all $1\leq p<\ii$.
\end{cor}

\section{A Wiener-Wintner-type ergodic theorem with Besicovitch weights}
The goal of this section is to show that Theorem \ref{t32} remains valid if one expands the set $\big\{ \{\lb^k\}: \,\lb\in\mbb C_1\big\}$ to the set of all bounded Besicovitch sequences.

A function $P : \mbb Z \to \mbb C$ is called a {\it trigonometric polynomial} if $P(k)=\sum\limits_{j=1}^s z_j\lb_j^k$, $k\in\mbb Z$, for some $s\in\mbb N$, $\{ z_j \}_1^s \su\mbb C$, and $\{ \lb_j \}_1^s\su\mbb C_1$. A sequence $\{b_k\}_{k=0}^\ii\su\Bbb C$ is called a {\it bounded Besicovitch sequence} if

(i) $|b_k | \leq C<\ii$ for all $k$ and some $C>0$;

(ii) for every $\ep >0$ there exists a trigonometric polynomial $P$ such that
\[
\limsup_n \frac 1n \sum_{k=0}^{n-1} |b_k - P(k) |<\ep .
\]

By linearity, Corollary \ref{c31} implies the following.
\begin{pro}\label{p41}
Let $f\in\mc L^1$. Then for every $\ep>0$ there exists $\Om^\pr\su\Om$ with $\mu(\Om\sm\Om^\pr)\leq\ep$ such that the sequence
\[
M_n(T,P)(f)\chi_{\Om^\pr}=\frac1n\sum_{k=0}^{n-1}P(k)f\circ T^k\chi_{\Om^\pr}, \ n=1,2,\dots
\]
converges uniformly for any trigonometric polynomial $P=P(k)$.
\end{pro}

Let us denote by $\mc B$ the set of Besicovitch sequences. The next theorem is an extension of Theorem \ref{t31}.
\begin{teo}\label{t41}
If $f\in\mc L^1$, then for any $\ep>0$ there is $\Om^\pr\su\Om$ with $\mu(\Om\sm\Om^\pr)\leq\ep$ such that the sequence $\{M_n(T,\ol b)(f)\chi_{\Om^\pr}\}$ converges uniformly for every \,$\ol b\in\mc B$.
\end{teo}
\begin{proof}
In view of Proposition \ref{p11}, it is sufficient to show that the convergence holds for any $f\in\mc L^1\cap\mc L^\ii$. So, pick $0\neq f\in\mc L^1\cap\mc L^\ii$ and let $\ep>0$, $\dt>0$. By Proposition \ref{p41}, there exists $\Om^\pr\su\Om$ with $\mu(\Om\sm\Om^\pr)\leq\ep$ such that for any trigonometric polynomial $P=P(k)$ there is $N_1=N_1(P)$ satisfying
\[
\|(M_m(T,P)(f)-M_n(T,P)(f))\chi_{\Om^\pr}\|_\ii\leq\frac\dt3 \ \ \,\forall \ m,n\ge N_1.
\]
Let $\ol b=\{b_k\}\in\mc B$, and let a trigonometric polynomial $P=P(k)$ be such that
\[
\limsup_n\frac1n\sum_{k=0}^{n-1}|b_k-P(k)|<\frac\dt{3\|f\|_\ii}.
\]
Then there exists $N_2$ such that \,$\frac1n\sum_{k=0}^{n-1}|b_k-P(k)|<\frac\dt{3\|f\|_\ii}$ \,whenever $n\ge N_2$. Now, if $m,n\ge\max\{N_1,N_2\}$, it follows that
\[
\begin{split}
\|(M_m&(T,\ol b)(f)-M_n(T,\ol b)(f))\chi_{\Om^\pr}\|_\ii\leq\|M_m(T,\ol b)(f)-M_m(T,P)(f)\|_\ii\\
&+\|M_n(T,\ol b)(f)-M_n(T,P)(f)\|_\ii+\|(M_m(T,P)(f)-M_n(T,P)(f))\chi_{\Om^\pr}\|_\ii\\
&\leq2\|f\|_\ii\frac1n\sum_{k=0}^{n-1}|b_k-P(k)|+\frac\dt3<\dt,
\end{split}
\]
so, the sequence $\{M_n(T,\ol b)(f)\chi_{\Om^\pr}\}$ converges uniformly for each $\ol b\in\mc B$.
\end{proof}

As in Theorem \ref{t32}, we derive the following.
\begin{teo}\label{t42}
Theorem \ref{t41} holds for all $f\in\mc R_\mu$.
\end{teo}

\begin{cor}
Given $1\leq p<\ii$, Theorem \ref{t41} holds for all $f\in\mc L^p$.
\end{cor}

Note that when $\mu(\Om)=\ii$, there are functions in $\mc R_\mu$ that do not belong to any of the spaces $\mc L^p$, $1\leq p<\ii$, but lie in a classical function Banach space $X$ such as Orlicz, Lorentz, or Marcinkiewicz with $\mathbf 1=\chi_\Om\notin X$. If $\mathbf 1\notin X$, then $X\su\mc R_\mu$ by \cite[Proposition 6.1]{ccl}, hence Theorems \ref{t32} and \ref{t42} hold for any $f\in X$. For conditions that warrant $\mathbf 1\notin X$ when $X$ is an Orlicz, Lorentz, or Marcinkiewics space, that is, for applications of Theorems \ref{t32} and \ref{t42} to these spaces, see \cite[Section 5]{cl}.

\end{document}